\documentclass[10pt, reqno]{amsart}
\usepackage{amssymb}
\usepackage{latexsym}
\usepackage{amsmath}
\usepackage{amsthm}
\usepackage{amsfonts}
\usepackage{epsfig}
\usepackage{verbatim}
\usepackage[all,cmtip]{xy}
\usepackage[colorlinks=true,linkcolor=blue,citecolor=blue]{hyperref}

\numberwithin{equation}{section}

\newtheorem{theorem}{Theorem}[section]
\newtheorem*{theorem*}{Theorem}

\newtheorem{lemma}[theorem]{Lemma}
\newtheorem*{lemma*}{Lemma}

\newtheorem{corollary}[theorem]{Corollary}
\newtheorem*{corollary*}{Corollary}

\newtheorem{proposition}[theorem]{Proposition}
\newtheorem*{proposition*}{Proposition}

\theoremstyle{definition}
\newtheorem{definition}[theorem]{Definition}

\newtheorem{remark}[theorem]{Remark}
\newtheorem*{remark*}{Remark}

\theoremstyle{remark}

\newcommand{\slashed}[1]{#1\kern-0.65em /}

\newcommand{\A}{\mathcal A}
\newcommand{\B}{\mathcal B}

\newcommand{\sflow}{\textup{SF}}

\newcommand{\Id}{\textup{Id}}

\newcommand{\End}{\textup{End}}

\newcommand{\ind}{\textup{Ind}}
\newcommand{\cl}{\textup{c}}
\newcommand{\ldirac}{\widehat D}
\newcommand{\ahat}{\hat A}
\newcommand{\even}{\textup{even}}
\newcommand{\odd}{\textup{odd}}
\newcommand{\reta}{\xi}

\newcommand{\tr}{\text{tr}}

\newcommand{\Ch}{\textup{Ch}}
\newcommand{\sch}{ \widetilde{Ch}}
\newcommand{\tch}{{\textup{Tch}}}

\begin{document}

\title[Relative Index Pairing and Odd Index Theorem]{Relative Index Pairing and Odd Index Theorem for Even Dimensional Manifolds}
\author{Zhizhang Xie}
\address{Department of Mathematics, The Ohio State University,
Columbus, OH, 43210-1174, USA}
\email{xiezz@math.ohio-state.edu}


\begin{abstract}
We prove an analogue for even dimensional manifolds of the Atiyah-Patodi-Singer twisted index theorem for trivialized flat bundles. We show that the eta invariant appearing in this result coincides with the eta invariant by Dai and Zhang up to an integer. We also obtain the odd dimensional counterpart for manifolds with boundary of the relative index pairing by Lesch, Moscovici and Pflaum. 
\end{abstract}
\thanks
{ The author was partially supported by the US National Science Foundation awards no. DMS-0652167.}	

\maketitle

\section*{Introduction}

In this article, we will prove an analogue for even dimensional manifolds of the Atiyah-Patodi-Singer twisted index theorem for trivialized flat bundles over odd dimensional closed manifolds \cite[Proposition 6.2]{A-P-S76}, and some related results. For notational simplicity, we will restrict the discussion mainly to spin manifolds. However all results can be straightforwardly extended to general manifolds. Unless we specify otherwise, we always fix the Riemannian metric for each manifold in this article and use the associated Levi-Civita connection to define its characteristic classes.

To motivate the subject matter of this paper, we begin by recalling the APS twisted index theorem for odd dimensional closed manifolds in the following form, cf. \cite[Corollary 7.9]{L-M-P09}. For   $(p_s)_{0\leq s \leq 1} \in M_k(C^\infty(N))$, $s\in [0, 1]$, a smooth path of projections over $N$, one has

\begin{equation*}
	\int_0^1 \frac{1}{2} \frac{d}{ds}  \eta(p_s D p_s)ds  = \int_N \ahat(N) \wedge \tch_\bullet(p_s).
\end{equation*}
Here $p_s D p_s$ is the Dirac operator twisted by $p_s$, $\eta(p_sD p_s)$ its $\eta$-invariant,  $\ahat(N)$ the $\ahat$-genus form of $N$ and $\tch_\bullet(p_s)$ is the Chern-Simons transgression form of $(p_s)_{0\leq s \leq 1}$, cf. Section $ \ref{sec:chern}$.

To prove our analogue for even dimensional closed manifolds, we shall replace a path of projections by a path of unitaries. The more interesting issue is what should replace the $\eta$-invariant appearing on the left hand side of the above formula. To answer this, let us first consider the case where the manifold in question bounds, that is, it is the boundary of some spin manifold. In this case, the $\eta$-invariant by Dai and Zhang \cite[Definition $2.2$]{XD-WZ06} is the right candidate, cf. Section $ \ref{sec:dz}$ below. Indeed, suppose the even dimensional manifold $Y$ is the boundary of a spin manifold $X$ and $(U_s)_{0\leq s \leq 1}$ is the restriction to $Y$ of a smooth path of unitaries over $X$. Denote the $\eta$-invariant of Dai and Zhang by $\eta(Y, U_s)$ for each $s\in [0,1]$, then  
\begin{equation}\label{eq:main-dz}
\int_0^1 \frac{1}{2}\frac{d}{ds} \eta(Y, U_s )ds  = \int_Y \ahat(Y) \wedge \tch_\bullet(U_s) .
\end{equation} 
When $Y$ bounds, it follows from the cobordism invariance of the index of Dirac operators that $\ind(D^+) = 0$, where $D^+$ is the restriction of the Dirac operator over $Y$ to the even half of the spinor bundle according to its natural $\mathbb Z_2$-grading. The condition $\ind(D^+)= 0$ is crucial for the definition of the $\eta$-invariant by Dai and Zhang, however is often not satisfied by even dimensional closed spin manifolds in general. To cover the general case, we shall use another approach where we lift the data to $\mathbb S^1\times Y$. The main ingredient of the method of proof is using an explicit formula of the cup product $K^1(\mathbb S^1) \otimes K^1(Y) \to K^0(\mathbb S^1\times Y)$, inspired by the Powers-Rieffel idempotent construction, cf. \cite{MR81}. In fact, the formula given for the case when $Y=\mathbb S^1$ by Loring in \cite{TL88} also works for all manifolds in general, cf. Section $ \ref{sec:ep}$ below. Our analogue for even dimensional closed spin manifolds of the APS twisted index theorem (Theorem $\ref{Thm:main}$ below) is as follows.
\begin{theorem*}[\textbf{I}]
Let $Y$ be an even dimensional closed spin manifold and $(U_s)_{0\leq s \leq 1} \in U_k(C^\infty(Y))$ a smooth path of unitaries over $Y$. For $s\in [0, 1]$, $e_s\in M_{2k}(C^\infty(\mathbb S^1\times Y))$ is the projection defined as the cup product of $U_s$ with the generator $e^{2\pi i \theta}$ of $K^1(\mathbb S^1)$. Let  $D_{\mathbb S^1\times Y}$ be the Dirac operator over $\mathbb S^1\times Y$. Then
\begin{equation}\label{eq:main-circ}
  \int_0^1 \frac{1}{2}\frac{d}{ds} \eta(e_s D_{\mathbb S^1\times Y} e_s)ds  = \int_Y \ahat(Y) \wedge \tch_\bullet(U_s).
\end{equation}
\end{theorem*}
The formula for $e_s$ is given in Section $\ref{sec:ep}$. A priori, the $\eta$-invariants in the formulas $\eqref{eq:main-dz}$ and $\eqref{eq:main-circ}$ appear to be different, we however will show that they are equal to each other modulo $\mathbb Z$ (Theorem $\ref{thm:eta}$ below) in the case where $Y$ bounds. 
\begin{theorem*}[\textbf{II}]
Suppose $Y$ is the boundary of an odd dimensional spin manifold. If $U\in U_k(C^\infty(Y))$ is a unitary over $Y$ and $e_U$ is the cup product of $U$ with $[e^{2 \pi i\theta}]\in K^1(\mathbb S^1)$, then one has
\[\eta( Y, U) = \eta(e_U D_{\mathbb S^1\times Y}  e_U) \mod \mathbb Z. \]

\end{theorem*}
The method of proof is based on a slight generalization of a theorem by Br\"{u}ning and Lesch \cite[Theorem 3.9]{JB-ML99}, see Proposition $\ref{prop:path}$ below. In this sense, $\eta(e_U D_{\mathbb S^1\times Y} e_U)$ can be thought of as the extension to general even dimensional manifolds of the definition of the $\eta$-invariant by Dai and Zhang. 

The same technique used above also allows us to prove the following analogue (Theorem $\ref{Thm:oddpairing}$ below) for odd dimensional manifolds with boundary of the relative index pairing formula by Lesch, Moscovici and Pflaum \cite[Theorem $7.6$]{L-M-P09}. Suppose $M$ is an odd dimensional spin manifold with boundary $\partial M$. By a relative $K$-cycle $[U, V, u_s]\in K^1(M, \partial M)$, we mean $U, V\in U_n(C^\infty(M))$ are two unitaries over $M$ with $u_s \in U_n(C^\infty(\partial M))$, $s\in [0, 1]$, a smooth path of unitaries over $\partial M$ such that $u_0 = U|_{\partial M} $ and $u_1 = V|_{\partial M}$. We denote by $T_U$, resp. $T_V$, the Toeplitz operator on $M$ with respect to $U$, resp. $V$ (see Section $\ref{sec:eta}$ for details). 
\begin{theorem*}[\textbf{III}]
	Let $[U, V, u_s]$ be a relative $K$-cycle in $ K^1(M, \partial M)$. If $U$ and $V$ are constant along the normal direction near the boundary, then 
	\[  \ind_{[D]}( [U,V, u_s]) = \ind(T_V) -\ind(T_U) +\sflow \left(u_s^{-1}D_{[0,1]} u_s; P_0^{u_s}\right)\]
	where $\sflow \left(u_s^{-1}D_{[0,1]} u_s; P_0^{u_s}\right)$ is the spectral flow of the path of elliptic operators $(u_s^{-1}D_{[0,1]} u_s; P_0^{u_s})$, $s\in [0, 1]$, with Atiyah-Patodi-Singer type boundary conditions determined by $P_0^{u_s}$ as in $\eqref{eq:path}$.
\end{theorem*}
This uses  Dai and Zhang's Toeplitz index theorem for odd dimensional manifolds with boundary \cite{XD-WZ06}. We shall give the details in Section $ \ref{sec:dz}$. 

It should be mentioned that, although the objects we work with are from classical geometry, the spirit of the proofs is very much inspired by methods from noncommutative geometry, cf. \cite{AC94}.

A brief outline of the article is as follows. In Section $ \ref{sec:pre}$, we recall some results about index pairings for manifolds with boundary. Section $ \ref{sec:ep}$ is devoted to the explicit formula of the cup product in $K$-theory mentioned earlier. This allows us to carry out explicit calculations for Chern characters in Section $ \ref{sec:chern}$. With these preparations, we prove an analogue for even dimensional manifolds of the APS twisted index theorem in Section $ \ref{sec:odd}$. In Section $ \ref{sec:eta}$, we show the equality of the two a priori different eta invariants. In the last Section, we prove the odd-dimensional counterpart of the relative index pairing formula by Lesch, Moscovici and Pflaum \cite[Theorem $7.6$]{L-M-P09}.

	
\subsection*{Acknowledgements}

I am greatly indebted to Henri Moscovici for his continuous support and advice. This paper grew out of numerous conversations with him. I want to thank Nigel Higson for helpful suggestions. I am grateful to Alexander Gorokhovsky for a careful reading of the first version of this paper as well as for many helpful comments. I started working on this problem during my visit at the Hausdorff Center for Mathematics in Bonn, Germany. I want to express my thanks to the center for its hospitality and to Matthias Lesch for the invitation, as well as for generously sharing with me his insights into the subject.

\section{Relative Index Pairing}\label{sec:pre}

Let $M$ be a compact smooth manifold with boundary $\partial M\neq \emptyset$. Following \cite[Sec. 2]{B-D-T89}, consider an elliptic first order differential operator
\[
D : C^\infty_c(M\setminus\partial M, E) \to C^\infty_c(M\setminus\partial M, E)
\]
where  $C^\infty_c(M\setminus\partial M, E)$ is the space of compactly supported smooth sections of the Hermitian vector bundle $E$. Such an operator has a number of extensions to become a closed unbounded operator on $H = L^2(M\setminus\partial M, E)$, e.g. $D_{min}$ and $D_{max}$ the minimum extension and the maximum extension respectively. Consider $D_e$  a closed extension of $D$ such that
\begin{equation}\label{Eq:domain}
	 D_{min} \subset D_e \subset D_{max}, 
\end{equation}
that is, $\mathcal D(D_{min}) \subset \mathcal D(D_e) \subset \mathcal D(D_{max})$.
Let
\[ B = \begin{pmatrix} 0 & D_e^\ast \\ D_e & 0 \end{pmatrix} \]
and
\[ F_e = B(B^2 + 1)^{-1/2} = \begin{pmatrix} 0 & T^\ast \\ T & 0 \end{pmatrix} \] 
with $T = D_e(D_e^\ast D_e + 1)^{-1/2} $ and $T^\ast = D_e^\ast(D_eD^\ast_e + 1)^{-1/2}$. Denote by $C_0(M\setminus\partial M)$ the space of continuous functions vanishing at infinity. Then the $\ast$-representation of $C_0(M\setminus\partial M)$ on $H\oplus H$ given by scalar multiplication, together with $F_e$, defines an element in $KK(C_0(M\setminus\partial M), \mathbb C)$, see \cite{B-D-T89} for the precise construction. Such a $K$-homology class turns out to be independent of the choice of a closed extension of $D$ \cite[Proposition 2.1 ]{B-D-T89}, and will be denoted $[D]$. Similarly for each formally symmetric elliptic operator, one constructs a cycle in $KK(C_0(M\setminus\partial M), \textup{Cl}_1)$ \cite[Section 2]{B-D-T89}, where $\textup{Cl}_1$ is the Clifford algebra with one generator. For each $[D]\in KK(C_0(M\setminus \partial M), \textup{Cl}_\bullet)$, one has the index pairing map
\[ \ind_{[D]} : K^\bullet(M\setminus \partial M) \to \mathbb Z.\]  

An element in $K^0(M\backslash \partial M)$ is represented by a triple $(E, F, \alpha)$ with $E$, $F$ vector bundles over $M\backslash \partial M$ and $\alpha: E\to F$ a bundle homomorphism whose restriction near infinity is an isomorphism, cf. \cite{MA67}. Moreover, we can choose connections over the bundles $E$ and $F$ such that the forms $\Ch_\bullet(E) $ and $\Ch_\bullet(F)$ coincide near infinity. Under this assumption, one can write down an explicit formula for the index pairing map:
\begin{equation*}
\ind_{[D]}([E, F, \alpha])= \int_{M} \omega_{D} \wedge \left[\Ch_\bullet(E)-\Ch_\bullet(F) \right]. 
\end{equation*}
Here $\omega_D\in H^{\textup{even}}_{\textup{dR}}(M\backslash \partial M)$ is the dual of the Chern character of the $K$-homology class $[D]$, as explained in the introduction of Chap. I in \cite{AC85}. In the case where $M$ is a spin manifold and $D$ the Dirac operator over $M$, one has $\omega_D = \ahat(M)$. 

Similarly, in the odd case, an element in $K^1(M\setminus \partial M)$ consists of two unitaries $U$ and $V$ over $M\setminus \partial M$ and a homotopy $h$ between $U$ and $V$ near infinity. Moreover, we can assume that $U$ and $V$ are identical near infinity and the homotopy $h$ becomes the identity map near infinity, cf. e.g.\cite[Prop. 4.3.14]{NH-JR00}, in which case the index pairing map has the following cohomological expression \footnote{We adopt the negative sign here in order to be consistent with our sign convention throughout the article.}:
\begin{equation*}\label{Eq:orip}
\ind_{[D]}( [V, U, h]) = -\int_{M} \omega_{D} \wedge \left[\Ch_\bullet(V)-\Ch_\bullet(U) \right]. 
\end{equation*}

Note that the boundary data are conspicuously absent in the above formulas. Indeed, by definition, $K^\bullet(M\backslash \partial M)$ is essentially the (reduced) $K$-group of the one point compactification of $M\backslash \partial M$. The information from the boundary is therefore completely eliminated from the picture. In order to recover that, we shall turn to the relative $K$ theory of the pair $(M, \partial M)$, denoted  $K^\bullet(M, \partial M)$, cf. \cite{L-M-P09}. A relative $K$-cycle $[p, q, h_s]\in K^0(M, \partial M)$ is a triple where $p, q \in M_n(C^\infty(M))$ are two projections over $M$ and $h_s \in M_n(C^\infty(\partial M))$, $s\in [0,1]$, is a path of projections over $\partial M$ such that $h_0 = p|_{\partial M}$ and $h_1 = q|_{\partial M}$. Similarly,  a relative $K$-cycle $[U, V, u_s]\in K^1(M, \partial M)$ is a triple where $U, V\in U_n(C^\infty(M))$ are two unitaries over $M$ with $u_s \in U_n(C^\infty(\partial M))$, $s\in [0, 1]$, a smooth path of unitaries over $\partial M$ such that $u_0 = U|_{\partial M} $ and $u_1 = V|_{\partial M}$. First notice that $K^\bullet(M, \partial M)\cong K^\bullet(M\setminus \partial M) $. Hence the above index pairing induces a map $\ind_{[D]}: K^\bullet(M, \partial M) \to \mathbb Z$. The issue now is to find an explicit formula which incorporates geometric information of the boundary. For even dimensional manifolds with boundary, this is done by Lesch, Moscovici and Pflaum\cite[Theorem $7.6$]{L-M-P09}. We shall give an analogous formula for odd dimensional manifolds with boundary in the Section $ \ref{sec:dz}$.

\section{Cup Product in K-theory}\label{sec:ep}

Let $\mathcal A$  and  $ \B$  be local Fr\'echet algebras. The cup product between $K_1(\A)$ and $K_1(\B)$ is defined by
\[\times :  K_1(\B) \otimes K_1(\A) = K_0(S\B) \otimes K_0(S\A) \to K_0(S\B \otimes S\A) \cong K_0(\B\otimes \A) \] 
where $S\A$ (resp. $S\B $)  is the suspension of $\A$ (resp. $S\B$), the isomorphism is the Bott isomorphism and
\[K_0(S\B) \otimes K_0(S\A) \to K_0(S\B \otimes S\A) \] is given by 
\begin{equation}\label{Eq:ep}
	 [p] \times [q] = [p\otimes q].
\end{equation}

In the case where $\B = C^\infty(\mathbb S^1)$, we shall give an explicit formula of this cup product. Since $e^{2\pi i \theta}$ is a generator of $K_1(C^\infty(\mathbb S^1)) \cong \mathbb Z$, it suffices to give this formula for $[e^{2\pi i \theta}] \times [U] $ with $U\in U_k(\A)$.


\begin{lemma}[cf. \cite{TL88}]\label{Lemma:ep} With the above notation,
\[	
[e^{2\pi i \theta}] \times [U] = [e_{U}]
\]
where $e_U =\begin{pmatrix} f & g+ h U \\ h U^\ast +g & 1-f \end{pmatrix} \in M_{2k}(C^\infty(\mathbb S^1)\otimes \mathcal A)$ is a projection with $f, g$ and $h$ nonnegative functions on $\mathbb S^1$ satisfying the following conditions

\begin{enumerate}
	\item $0\leq f \leq 1$,
	\item $f(0) = f(1) = 1$ and $f(1/2) = 0$,
	\item $g = \chi_{[0, 1/2]} (f - f^2)^{1/2}$ and $h = \chi_{[1/2, 1]}(f- f^2)^{1/2}$.
\end{enumerate}
	
\end{lemma}

\begin{proof}
It is not difficult to see that 
\[\times: K_1(C^\infty(\mathbb S^1)) \otimes K_1(\mathcal A) \to K_0(C^\infty(\mathbb S^1) \otimes \mathcal A)\]
 is the same as the standard isomorphism \cite[Section 7.2]{WO93}  
 \[\Theta_{\A} : K_1(\A) \to K_0(S\A) \subset K_0(C^\infty(\mathbb S^1)\otimes \A)\]
after identifying $K_1(C^\infty(\mathbb S^1))$ with $ \mathbb Z$. The inverse of this map is constructed as follows, cf. \cite[Proposition 4.8.2]{NH-JR00}\cite[Section 7.2]{WO93}. The group $K_0(S\A)$ is generated by formal differences of normalized loops of projections over $\A$. Such a loop is a projection-valued maps $p : [0,1] \to M_n(\A)$ with $p(0) = p(1)\in M_n(\mathbb C)$. For each loop, there is a path of unitaries $u : [0, 1] \to U_n(\A)$ such that $p(t) = u(t)p(1)u(t)^\ast$ and $u(0) = 1_n$.Without loss of generality, we can assume $p(0) =p(1) = \begin{pmatrix} 1 & 0 \\ 0 & 0\end{pmatrix}$. This implies that $u(1)$ is of the form $ \begin{pmatrix} v & 0 \\ 0 & w \end{pmatrix}$. Then one checks that $[p] \mapsto [v]$ is a well-defined inverse to $\Theta_{\A}$. 

To see that our formula agrees with the usual definition, it suffices to show that $\Theta_{\A}^{-1}(e_U) = U  $. First notice that $e_U(0) = e_U(1) = \begin{pmatrix} 1 & 0 \\ 0 & 0\end{pmatrix}$ and $e_U(\theta)$ is a projection over $\A$ for each $\theta\in \mathbb S^1 = \mathbb R/\mathbb Z$, hence $e_U$ is a normalized loop of projections. Now consider the following path of unitaries over $\A$, 
\[ \mathcal U(\theta) = \begin{pmatrix} f_1(\theta) + f_2(\theta) U & (1-f)^{1/2}(\theta) \\ (1-f)^{1/2}(\theta)  & -f_1(\theta) - f_2(\theta) U^\ast \end{pmatrix}\]
where $f_1= \chi_{[0, 1/2]} f^{1/2}$ and $ f_2 = \chi_{[1/2, 1]} f^{1/2}$. In particular, $\mathcal U(0) = \begin{pmatrix} 1 & 0 \\ 0 & 1\end{pmatrix}$ and $\mathcal U(1) = \begin{pmatrix} U & 0 \\ 0 & -U^\ast\end{pmatrix}$.
By a direct calculation, one verifies
\[   
e_U (\theta) = \mathcal U(\theta) \begin{pmatrix} 1 & 0 \\ 0 & 0\end{pmatrix} \mathcal U(\theta)^\ast
\]
from which the lemma follows.

\end{proof}

We will also make use of the following lemma in Section $ \ref{sec:chern}$, cf. \cite[Lemma $2.2$]{TL88}.  
\begin{lemma} \label{Lemma:id}
For $f, g$ and $h$ nonnegative functions on $\mathbb S^1 = \mathbb R/\mathbb Z$ satisfying the following conditions
    \begin{enumerate}
		\item $0\leq f \leq 1$,
		\item $f(0) = f(1) = 1$ and $f(1/2) = 0$,
		\item $g = \chi_{[0, 1/2]} (f - f^2)^{1/2}$ and $h = \chi_{[1/2, 1]}(f- f^2)^{1/2}$,
	\end{enumerate}
we have 
\[ \int_0^1 \left[(2-4f)h'h^{2k-1} + 4f'h^{2k} \right] d\theta = \frac{(k-1)! (k-1)!}{(2k-1)!}\]
\end{lemma}
\begin{proof}
	Notice that
	\[\int_0^1 f'h^{2k} d\theta = \int_{1/2}^1 (f(\theta)-f^2(\theta))^k df(\theta) = \int_0^1 (x-x^2)^k dx = \frac{k! k!}{(2k+1)!},\]
and integration by parts gives 
	\[\int_0^1 (2-4f)h'h^{2k-1} d\theta = \frac{2}{k} \int_{0}^1  f'h^{2k} d\theta.\]
\end{proof}

\section{Chern Characters and Transgression Formulas}\label{sec:chern}

Throughout this section, although we deal with commutative algebras, we shall use the similar formalism for the Chern character in $K$-theory as in cyclic homology \cite[Chap. II]{AC85}, \cite[Chap. VIII]{JL92}. 
Let $M$ be a compact smooth manifold with or without boundary. The even (resp. odd) Chern characters of projections (resp. unitaries) in $M_n(C^\infty(M))$ can be expressed as follows.  
	For $p\in M_n(C^\infty(M))$ such that $p^2 = p $ and $p^\ast = p$,  
		\begin{equation}\Ch_\bullet(p) := \tr(p) + \sum_{k=1}^{\infty}(-1)^{k}\frac{1}{(2\pi i)^{k}}\frac{1}{ k!} \tr\left( p (dp)^{2k}\right) \in H^{\even}_{\textup{dR}}(M).
		\end{equation}
	For $U \in U_n(C^\infty(M))$, 
	\begin{equation}\label{eq:oddchern}
	\Ch_\bullet(U) := \sum_{k=0}^\infty \frac{1}{(2\pi i)^{k+1}}\frac{k!}{(2k+1)!} \tr\left( (U^{-1}dU)^{2k+1}\right) \in H^{\odd}_{dR}(M).
	\end{equation} 







For each $U\in U_n(C^\infty(M))$, let $e_U$ be the projection as in Lemma $\ref{Lemma:ep}$. If no confusion is likely to arise, we also write $e$ instead of $e_U$.

\begin{lemma}\label{Lemma:ch}
	\[\Ch_\bullet(e_U) = -\sum_{k=1}^{\infty}\frac{1}{(2\pi i)^{k}}\frac{k}{k!} \left(4f'h^{2k}+ (2-4f)h'h^{2k-1}\right)d\theta \wedge\tr (U^{-1} \cdot dU)^{2k-1} \]
\end{lemma}

\begin{proof}
Notice that 
	\begin{align*}
	 de & =  \begin{pmatrix} f' & g' + h' U \\ h' U^\ast +g' & -f' \end{pmatrix}d\theta +\begin{pmatrix} 0 & h \ dU \\ h\ dU^\ast & 0 \end{pmatrix},
	\end{align*}
which implies
\begin{align} 
	\tr(e(de)^{2k}) &=  \tr \left( \begin{pmatrix} f& g + h U \\ h U^\ast +g & 1-f \end{pmatrix} \begin{pmatrix} 0 & h\ dU \\ h\  dU^\ast  & 0\end{pmatrix}^{2k} \right) \label{Eq:wot}\\
	& + \sum_{j=1}^{j=2k} \tr \left( \begin{pmatrix} f& g + h U \\ h U^\ast +g & 1-f \end{pmatrix} \begin{pmatrix} 0 & h\ dU \\ h\  dU^\ast  & 0\end{pmatrix}^{j-1} \right. \notag\\
	& \quad \left.  \begin{pmatrix} f' & g'+h' U \\g'+ h' U^\ast  & -f'\end{pmatrix} d\theta \begin{pmatrix} 0 & h\ dU \\ h\ dU^\ast  & 0\end{pmatrix}^{(2k-j)}  \right). \label{Eq:wt}
\end{align}
Since most of the matrices appearing in the above summation only have off diagonal entries, a straightforward calculation gives the following equalities.
\[  (\ref{Eq:wot})= h^{2k}\tr\left( (U^{-1} \cdot dU)^{2k} \right) \]
\[ 	 (\ref{Eq:wt}) =  -(-1)^{k} k \left((2-4f)h'h^{2k-1} +4f'h^{2k}\right)d\theta\wedge \tr(U^{-1} \cdot dU)^{2k-1}.\]
On the other hand, 
\[\tr ((U^{-1} \cdot dU)^{2k}) = -\tr\left((U^{-1} \cdot dU) (U^{-1} \cdot dU)^{2k-1} \right)\]
from which it follows that $\eqref{Eq:wot}$ vanishes. This finishes the proof.
	
\end{proof}

As a consequence of lemma $(\ref{Lemma:id})$ and lemma $(\ref{Lemma:ch})$, one has the following corollary. From now on,  integration along the fiber $\mathbb S^1$ will be denoted by $\pi_\ast$. 
\begin{corollary}\label{cor:ch} 
$\pi_\ast\Ch_\bullet(e_U) = - \Ch_\bullet (U). $ 
\end{corollary}

Consider a smooth path of unitaries $U_s\in U_n(C^\infty(M))$ with $s\in [0, 1]$, or equivalently $\mathfrak U\in U_n(C^\infty([0, 1]\times M))$. The secondary Chern character $\sch_\bullet(U_s)$ is given by the formula
	\[ \sch_\bullet(U_s) := \sum_{k=0}^{\infty}(-1)^k\frac{1}{(2\pi i)^{k+1}}\frac{k!}{(2k)!}\tr\left(U_s^{-1}\dot{U_s} (U_sdU^{-1})^{2k}\right). \]
Then $\Ch_\bullet(\mathfrak U)$ can be decomposed as 
\[ \Ch_\bullet(\mathfrak U) = \Ch_\bullet(U_s) + ds\wedge \sch_\bullet(U_s) \] 
where $\Ch(U_s)$ (see $\eqref{eq:oddchern}$ above)  and $\sch_\bullet(U_s)$ do not contain $ds$. Applying de Rham differential to both sides gives us the following transgression formula
\[ \frac{\partial}{\partial s} \Ch_\bullet(U_s) = d\sch_\bullet(U_s). \]
Similarly, if $e_s \in M_m(C^\infty(M)) $ is a smooth path of projections, or equivalently a projection $\mathfrak e \in M_m(C^\infty([0,1]\times M))$, then
\[ \Ch_\bullet(\mathfrak e) = \Ch_\bullet (e_s) + ds\wedge \sch_\bullet(e_s). \] 
with
\begin{align*}
	&\sch_\bullet (e_s) := \sum_{k=0}^\infty(-1)^{k+1}\frac{1}{(2\pi i)^{k+1}}\frac{1}{k!}\tr\left((2e_s-1)\dot e_s(de_s)^{2k+1}\right).
\end{align*}
Applying Corollary $(\ref{cor:ch})$ to $\Ch_\bullet(\mathfrak U)$ and $\Ch_\bullet(\mathfrak {e_U})$, one has
\[ \pi_\ast\Ch_\bullet(\mathfrak {e_U}) = -  \Ch_\bullet(\mathfrak U), \]
which implies 
\[ ds\wedge\pi_\ast\sch_\bullet(e_s ) = ds\wedge \sch_\bullet(U_s). \]
Denote the Chern-Simons transgression forms by
\[ \tch_\bullet(e_s)_{0\leq s\leq 1} := \int_0^1 ds\wedge \sch_\bullet(e_s),\] 
\[  \tch_\bullet\left(U_s\right)_{0\leq s\leq 1} := \int_0^1 ds\wedge \sch_\bullet(U_s). \]
We summarize the results of this section in the following proposition.
\begin{proposition}\label{prop:product}
Consider $U\in U_n(C^\infty(M))$ and $U_s\in U_n(C^\infty(M))$ for $ s\in [0,1]$. Let $e$, resp.  $e_s$, be the cup product of $U$, resp. $U_s$, with $e^{2 \pi i \theta}$ a generator of $K^1(\mathbb S^1)$ as in Lemma $\ref{Lemma:ep}$. Then 
\[   \pi_\ast\Ch_\bullet(e ) = - \Ch_\bullet(U) \]
and
\[  \pi_\ast\tch_\bullet(e_s )_{0\leq s\leq 1} = \tch_\bullet(e_s)_{0\leq s\leq 1}. \]

\end{proposition}

\section{Odd Index Theorem on Even Dimensional Manifolds}\label{sec:odd}

In this section, we shall prove our analogue for even dimensional closed manifolds of the APS twisted index theorem.

 Let us first recall the APS twisted index theorem and fix some notation. Let $N$ be a closed odd dimensional spin manifold and $\slashed D$ its Dirac operator. If $p$ is a projection in $M_n(C^\infty(N))$, then $p$ induces a Hermitian vector bundle, denoted $E_p$, over $N$. With the Grassmannian connection on $E_p$, let $p(D\otimes I_n)p$ be the twisted Dirac operator with coefficients in $E_p$. For notational simplicity, we also write $p D p$ instead of $p(D\otimes I_n)p$. 
Then by \cite[Corollary $7.9$]{L-M-P09}, for  $p_s\in M_k(C^\infty(N))$ a smooth path of projections over $N$, one has
\begin{equation}\label{Eq:gaps}
	 \reta(p_1 \slashed D p_1) -\reta(p_0 \slashed Dp_0) = \int_N \ahat(N) \wedge \tch_\bullet(p_s) + \sflow (p_s\slashed D p_s)_{0\leq s \leq 1} 
\end{equation}
where \[\reta(p_i\slashed D p_i)  = \frac{\eta(p_i \slashed D p_i) + \dim\ker(p_i \slashed D p_i)}{2}\] the reduced eta invariant of $p_i \slashed D p_i$. Here $\sflow (p_s\slashed D p_s)_{0\leq s \leq 1}$ is the spectral flow of $(p_s \slashed D p_s)_{0\leq s\leq 1}$. Notice that the vector bundle on which $p_s \slashed D p_s$ acts may vary as $s$ moves along $[0, 1]$. To make sense of the definition of such a spectral flow, we introduce a path of unitaries $u_s\in U_n(C^\infty(N))$ over $N$ with $u_s p_0 u_s^\ast = p_s$ so that $p_0 u_s^\ast \slashed D u_s p_0$ acts on the same vector bundle $E_{p_0}$. $ \sflow (p_s\slashed D p_s)_{0\leq s \leq 1} $ is then defined to be $\sflow (p_0 u_s^\ast\slashed D u_s p_0)_{0\leq s \leq 1}$ the spectral flow of the family $(p_0 u_s^\ast \slashed D u_s p_0)_{0\leq s \leq 1}$.  Now by \cite[Lemma 3.4]{PK-ML04}, formula $\eqref{Eq:gaps}$ is equivalent to
\begin{equation}\label{Eq:vareta} 
	\int_0^1 \frac{1}{2} \frac{d}{ds}  \eta(p_s\slashed D p_s)ds  = \int_N \ahat(N) \wedge \tch_\bullet(p_s).
\end{equation}


\begin{theorem}\label{Thm:main}
Let $Y$ be a closed even dimensional spin manifold and $(U_s)_{0\leq s \leq 1} \in U_k(C^\infty(Y))$ a smooth path of unitaries over $Y$. For $s\in [0, 1]$, let $e_s\in M_{2k}(C^\infty(Y))$ be the projection defined as the cup product of $U_s$ with the generator $e^{2\pi i \theta}$ of $K^1(\mathbb S^1)$. Let $D_{\mathbb S^1\times Y}$ be the Dirac operator over $\mathbb S^1\times Y$. Then
\begin{equation}
  \int_0^1 \frac{1}{2}\frac{d}{ds} \eta(e_s D_{\mathbb S^1\times Y} e_s)ds  = \int_Y \ahat(Y) \wedge \tch_\bullet(U_s).
\end{equation}
\end{theorem}
\begin{proof}
Applying formula $\eqref{Eq:gaps}$ to $\mathbb S^1\times Y$, one has 
\[ \reta(e_1 D_{\mathbb S^1\times Y} e_1) -\reta(e_0 D_{\mathbb S^1\times Y} e_0) = \int_{\mathbb S^1\times Y} \ahat(\mathbb S^1\times Y) \wedge \tch_\bullet(e_s) + \sflow (e_s D_{\mathbb S^1\times Y} e_s). \]
Notice that $\ahat (\mathbb S^1\times M) = \pi_1^\ast\ahat(\mathbb S^1)\wedge \pi_2^\ast\ahat(M)$ and $\ahat(\mathbb S^1) = 1$, where $\pi_1 : \mathbb S^1\times M \to \mathbb S^1$, resp. $\pi_2 : \mathbb S^1\times M \to M$, is the projection from $\mathbb S^1 \times M$ to $\mathbb S^1$, resp. $M$. By Proposition $\ref{prop:product}$, the integral on the right side is equal to $\int_{Y} \ahat(Y) \wedge \tch_\bullet(U_s).$ Now the formula 
\[  \int_0^1 \frac{1}{2}\frac{d}{ds} \eta(e_s D_{\mathbb S^1\times Y} e_s)ds  = \int_Y \ahat(Y) \wedge \tch_\bullet(U_s)\]
follows from the equality \cite[Lemma 3.4]{PK-ML04}
\[ \reta(e_1 D_{\mathbb S^1\times Y} e_1) -\reta(e_0 D_{\mathbb S^1\times Y} e_0) = \sflow (e_s D_{\mathbb S^1\times Y} e_s) + \int_0^1 \frac{1}{2}\frac{d}{ds} \eta(e_s D_{\mathbb S^1\times Y} e_s)ds . \]
\end{proof}

\begin{remark}
Mod $\mathbb Z$, the reduced $\eta$-invariant $\reta(e_s D_{\mathbb S^1\times Y} e_s)$ is equal to the reduced $\eta$-invariant $\reta(Y, U_s)$ defined by Dai and Zhang, cf. \cite[Definition 2.2]{XD-WZ06}, at least when $Y$ bounds. See Theorem $\ref{thm:eta}$ below.      
\end{remark}

\section{Equivalence of Eta Invariants}\label{sec:eta}

Throughout this section, we assume $M$ is an odd dimensional spin manifold with boundary $\partial M$. Denote by $\mathcal S_M$ the spinor bundle over $M$. Let  $D$ be the Dirac operator over $M$, then near  the boundary
\[ D = \cl(d/dx) \left( \frac{d}{dx} + D^\partial\right) \]
where $D^\partial$ is the Dirac operator over $\partial M$ and $c(d/dx)$ is the Clifford multiplication by the normal vector $d/dx$. Then $D\otimes I_n$ is the Dirac operator acting on $\mathcal S_M\otimes \mathbb C^n$, when we use the trivial connection on the bundle $M\times \mathbb C^n$ over $M$.  If no confusion is likely to arise, we shall write $D$ instead of $D\otimes I_n$.

Now a subspace $L$ of $\ker D^\partial$ is Lagrangian if $\cl(d/d x)L = L^{\perp} \cap \ker D^\partial$. In our case, since 	$\partial M$ bounds $M$, the existence of such a Lagrangian subspace follows from the cobordism invariance of the index of Dirac operators. Let $L_{>0}^2(\mathcal S_M\otimes \mathbb C^n|_{\partial M})$ be the positive eigenspace of $D^\partial$, i.e. the $L^2$-closure of the direct sum of eigenspaces with positive eigenvalues of $D^\partial$. Then the projection 
\[ P^\partial := P_{\partial M}(L) = P_{\partial M} + P_L \] 
imposes an APS type boundary condition for $D$, where $P_{\partial M}$, resp. $P_L$, is the orthogonal projection $L^2(\mathcal S_M\otimes \mathbb C^n|_{\partial M}) \to L_{>0}^2\mathcal (\mathcal S_M\otimes \mathbb C^n|_{\partial M})$, resp. $L^2(\mathcal S_M\otimes \mathbb C^n|_{\partial M})\to L$. Let us denote the corresponding self-adjoint elliptic operator by $D_{P^\partial} $. 

Let $L^2_{\geq 0}(\mathcal S_M\otimes \mathbb C^n; P^\partial)$ be the nonnegative eigenspace of $D_{P^\partial}$ and  $P_{P^\partial}$ the orthogonal projection 
\[P_{P^\partial} : L^2(\mathcal S_M\otimes \mathbb C^n) \to  L^2_{\geq 0}(\mathcal S_M\otimes C^n; P^\partial).\] 
More generally, for each unitary $U \in U_n( C^\infty(M))$ over $M$, the projection $UP^\partial U^{-1}$ imposes an APS type boundary condition for $D$ and we shall denote the corresponding elliptic self-adjoint operator by $D_{UP^\partial U^{-1}}$. Similarly, let $P_{UP^\partial U^{-1}}$ be the orthogonal projection 
\[P_{UP^\partial U^{-1}} : L^2(\mathcal S_M\otimes \mathbb C^n)\to L^2_{\geq 0}(\mathcal S_M\otimes \mathbb C^n; UP^\partial U^{-1})\] 
where $L^2_{\geq 0}(\mathcal S_M\otimes \mathbb C^n; UP^\partial U^{-1})$  is the nonnegative eigenspace of $D_{UP^\partial U^{-1}}$.

With the above notation, we define the Toeplitz operator on $M$ with respect to $U$ as follows, cf.\cite[Definition 2.1]{XD-WZ06}.
\begin{definition} 
	\[ T_U := P_{U P^\partial U^{-1}}\circ U \circ P_{P^\partial} .\]
\end{definition}

Dai and Zhang's index theorem for Toeplitz operators on odd dimensional manifolds with boundary \cite[Theorem 2.3]{XD-WZ06} states that
	\begin{equation}\label{Eq:toep}
	 \ind(T_U) = -\int_M \ahat (M) \wedge \Ch_\bullet(U) - \reta(\partial M, U) + \tau_{\mu}(UP^\partial U^{-1}, P^\partial, \mathcal P_M)  
	\end{equation}
	where $\mathcal P_M$ is the Calder\'on projection associated to the Dirac operator $D$ on $M$ (cf. \cite{BB-KW93}) and $\tau_{\mu}(UP^\partial U^{-1}, P^\partial, \mathcal P_M)$ is the Maslov triple index \cite[Definition 6.8]{PK-ML04}.
 The reduced $\eta$-invariant $\reta(\partial M, U)$ will be defined after the remarks.
\begin{remark}
	Notice that the integral in $(\ref{Eq:toep})$ differs from Dai and Zhang's by a constant coefficient $(2\pi i)^{-(\dim M+1)/2}$. This is due to the fact that our definition of characteristic classes follows topologists' convention, i.e. factors such as $(\frac{1}{2\pi i})^{k/2}$ are already included. 
\end{remark}

\begin{remark}
	The Maslov triple index $\tau_{\mu}(UP^\partial U^{-1}, P^\partial, \mathcal P_M)$ is an integer. For unitaries $U, V\in U_n(C^\infty(M))$, if there a path of unitaries $u_s\in  U_n(C^\infty(\partial M))$ with $s\in [0,1]$ such that  $u_0 = U|_{\partial M}$ and $u_1 = V|_{\partial M}$, one has
	\[ \tau_{\mu}(UP^\partial U^{-1}, P^\partial, \mathcal P_M) = \tau_{\mu}(VP^\partial V^{-1}, P^\partial, \mathcal P_M), \]
cf.\cite[Lemma 6.10]{PK-ML04}.
\end{remark}

To define $\reta(\partial M , U)$, let us first consider $D_{[0,1]} $ the Dirac operator over $[0,1]\times \partial M$.  If no confusion is likely to arise, we shall write $U$ for both $U|_{\partial M}$ and the trivial lift of $U|_{\partial M}$ from $\partial M$ to $[0, 1]\times \partial M$. Let
\begin{equation}  
	D^{\psi, U}_{[0,1]} :=  D_{[0,1]} + (1-\psi) U^{-1}[D_{[0,1]}, U ]
\end{equation}
over $[0,1]\times \partial M$, where $\psi$ is a cut-off function on $[0, 1]$ with $\psi \equiv 1$ near $\{ 0\} $ and $\psi \equiv 0$ near $\{ 1\}$. With APS type boundary conditions determined by $P^\partial $ on $\{0\}\times \partial M$ and $\Id - U^{-1}P^\partial U $ on $\{1\}\times \partial M$, $D^{\psi, U}_{[0,1]}$ becomes a self-adjoint elliptic operator, denoted 
$\left(D^{\psi, U}_{[0,1]}; P_0^U\right)$. See proposition $\ref{prop:path}$ for an explanation of the choice of notation.


Similarly, 
\begin{equation}\label{eq:vdirac}
	D^{\psi, U}_{[0,1]}(t) := D_{[0,1]} + (1-t\psi) U^{-1}\left[D_{[0,1]}, U\right].
\end{equation}
Denote by $\left(D^{\psi, U}_{[0,1]}(t); P_0^U\right)$ the elliptic operator $D^{\psi, U}_{[0,1]}(t)$ with boundary condition $P_0^U$. Note that $D^{\psi, U}_{[0,1]}(1) = D^{\psi, U}_{[0,1]}$. 
\begin{definition}(\cite[Definition 2.2]{XD-WZ06})
	\[ \overline \eta(\partial M, U) := \reta(D^{\psi, U}_{[0,1]}; P_0^U) - \sflow \left(D^{\psi, U}_{[0,1]}(t); P_0^U\right)_{0 \leq t\leq 1} \]
	where 
	\[  \reta(D^{\psi, U}_{[0,1]}) = \frac{\dim \ker (D^{\psi, U}_{[0,1]}; P_0^U) + \eta(D^{\psi, U}_{[0,1]}; P_0^U)}{2}  . \]
\end{definition}
\begin{remark}
$\overline \eta(\partial M, U)$ is independent of the cut-off function $\psi$ \cite[Proposition 5.1]{ XD-WZ06}.
\end{remark}


In order to show the equality $\reta(\partial M, U) = \reta(e_U D_{\mathbb S^1\times\partial M} e_U)\mod \mathbb Z$, we need to relate the operator $e_U D_{\mathbb S^1\times\partial M} e_U$ to $D^{\psi, U}_{[0,1]}$,  where $D_{\mathbb S^1\times\partial M} = \cl(d/{d\theta})\left(\frac{d}{d\theta} + D^\partial\right)$ is the Dirac operator over $\mathbb S^1\times \partial M$ and $e_U$ is the cup product of $U$ with $e^{2 \pi i\theta}\in K^1(\mathbb S^1)$. Recall that
\begin{align*}
	e_U D_{\mathbb S^1\times\partial M} e_U &= \begin{pmatrix} f & g + h U \\ h U^\ast +g & 1-f \end{pmatrix}\begin{pmatrix} D_{\mathbb S^1\times\partial M} & 0 \\ 0 & D_{\mathbb S^1\times\partial M} \end{pmatrix} \begin{pmatrix} f & g + h U \\ h U^\ast +g & 1-f \end{pmatrix} \\
&= \mathcal U \begin{pmatrix} 1 & 0\\ 0 & 0 \end{pmatrix} \mathcal U^\ast \begin{pmatrix} D_{\mathbb S^1\times\partial M} & 0 \\ 0 & D_{\mathbb S^1\times\partial M} \end{pmatrix} \mathcal U \begin{pmatrix} 1 & 0\\ 0 & 0 \end{pmatrix} \mathcal U^\ast 
\end{align*}
where 
\[ \mathcal U = \begin{pmatrix}  f_1^{1/2} + f_2^{1/2} U & (1-f)^{1/2} \\ (1-f)^{1/2}  &  - f_1^{1/2} - f_2^{1/2} U^\ast \end{pmatrix}\] 
with $f_1 = \chi_{[0, 1/2]} f$ and $f_2 = \chi_{[1/2, 1]} f$. Then viewed as an operator over $[0,1 ]\times \partial M$, 
\[ \mathcal U^\ast (e_U D_{\mathbb S^1\times\partial M} e_U) \mathcal U =  D_{[0,1]} + f_2 U^{-1}\left[ D_{[0,1]}, U\right] \]
with the boundary condition 
\[ \beta(0,x) = U \beta(1,x) , \quad  \textup{for} \ \forall  x\in \partial M \ \textup{ and } \ \beta\in \Gamma([0,1]\times \partial M ; S\otimes \mathbb C^n ).\]
Let $H^\partial :=  L^2(\{ 0\}\times \partial M ; S\otimes \mathbb C^n) \oplus L^2(\{1\}\times \partial M; S\otimes \mathbb C^n)$, then the above boundary condition can be written as 
\[ \frac{1}{2}\begin{pmatrix} 1 & -U \\ -U^{-1} & 1 \end{pmatrix} \beta = 0, \quad \textup{for}\ \forall \ \beta\in H^\partial.\]
From now on, let us assume $\psi = 1 - f_2$. In particular, one has
\[\mathcal U^\ast (e_U D_{\mathbb S^1\times\partial M} e_U) \mathcal U = D^{\psi, U}_{[0,1]}.\] 
Now consider 
\begin{equation}\label{eq:path}
 P_t^U = \begin{pmatrix} \cos^2t P^\partial + \sin^2t(I-P^\partial) & -\cos t\sin tU \\ -\cos t\sin tU^{-1} & \cos^2t (\Id - U^{-1}P^\partial U) + \sin^2tU^{-1}P^\partial U \end{pmatrix} 
\end{equation}
for $ 0\leq t \leq \pi/4$(cf.\cite[Equation 5.13]{PK-ML04}, \cite[Section 3]{JB-ML99}). This is a path of projections in $B(H^\partial)$ such that 
\[  P_0^U = \begin{pmatrix} P^\partial & 0 \\ 0 & \Id-U^{-1}P^\partial U \end{pmatrix} \]
and 
\[  P_{\pi/4}^U = \frac{1}{2}\begin{pmatrix} 1 & -U \\ -U^{-1} & 1 \end{pmatrix}. \]
For each $t\in [0, \pi/4]$, the Dirac operator $D_{[0,1]}^{\psi, U}$, with the boundary condition $P_t^U$, is a self-adjoint elliptic operator, denoted by $(D_{[0,1]}^{\psi, U}; P_t^U)$. 

With the above notation, we have the following slight generalization of a theorem by Br\"{u}ning  and Lesch \cite[Theorem 3.9]{JB-ML99} .
\begin{proposition}\label{prop:path}
\[ \frac{d}{dt}\eta(D_{[0,1]}^{\psi, U}; P_t^U) = 0.\] 
\end{proposition}
\begin{proof}
Following \cite[Section 3]{JB-ML99}, we define
\[\tau := \begin{pmatrix} 0 & U \\ U^{-1} & 0\end{pmatrix} = \begin{pmatrix} 0 & U \\ U^\ast & 0\end{pmatrix},  \]
\[ \widetilde\gamma := \begin{pmatrix} \cl(d/d\theta) & 0 \\ 0 & -\cl(d/d\theta)\end{pmatrix}, \]
\[\widetilde A  := \begin{pmatrix} D^\partial  & 0 \\ 0 & -U^{-1}D^\partial U \end{pmatrix}. \]
where $\widetilde A$ is determined by $D_{[0,1]}^{\psi, U}$ near the boundary, by noticing that  \[D_{[0,1]}^{\psi, U} =  \cl(d/{d\theta})\left(\frac{d}{d\theta} + D^\partial\right) \]
 near $\{0\} \times \partial M$ and 	
\[ D_{[0,1]}^{\psi, U} =  \cl(d/{d\theta})\left(\frac{d}{d\theta} + U^{-1}D^\partial U\right) \] 
near $\{1\} \times \partial M$.
Since $\cl(d/d\theta ) \ U = U \ \cl(d/d\theta)\in \End(\mathcal S\otimes \mathbb C^n)$, it follows that
\[\tau \widetilde A + \widetilde A \tau = 0 =\tau \widetilde\gamma + \widetilde \gamma\tau, \quad \tau^2 = 1, \quad \tau = \tau^\ast.\]   
Moreover, one verifies by calculation (cf. \cite[Eqs. (3.11) to (3.13)]{JB-ML99})
\begin{align*}
	& \widetilde \gamma P_t^U = (I-P_t^U)\widetilde\gamma;\\
	& [P_t^U, {\widetilde A}^2] = 0;\\
	& P_t^U \widetilde A P_t^U = \cos(2t) |\widetilde A|P_t^U.
\end{align*}
Then by \cite[Theorem 3.9]{JB-ML99}, it suffices to find a unitary $\mu: H^\partial \to H^\partial$ such that 
\[\mu^2 = -I, \quad \mu\tau + \tau \mu = \mu\widetilde \gamma + \widetilde\gamma \mu = \mu \widetilde A + \widetilde A\mu = 0 .\]
Let 
\[ \mu := \begin{pmatrix} 0 & U \\ -U^{-1} & 0\end{pmatrix}. \] This finishes the proof.
\end{proof}

Now the equality $\reta(\partial M, U) = \reta(e_U D_{\mathbb S^1\times \partial M} e_U) \mod \mathbb Z$ follows as a corollary. To be slightly more precise, we have the following result.

\begin{theorem}\label{thm:eta}
\[\reta(\partial M, U) = \reta(e_U D_{\mathbb S^1\times \partial M}  e_U) - \sflow(D_{[0,1]}^{\psi, U}; P_t^U) - \sflow (D_{[0,1]}^{\psi, U}(t); P_0^U)_{0\leq t \leq 1} . \]
In particular, 
\[ \reta(\partial M, U) = \reta(e_U D_{\mathbb S^1\times \partial M}  e_U) \mod \mathbb Z. \]
\end{theorem}

\begin{proof}
By \cite[Lemma 3.4]{PK-ML04},
\begin{align*} & \reta(e_U D_{\mathbb S^1\times \partial M}  e_U) -	\reta(D_{[0,1]}^{\psi, U}; P_0^U) \\
& = \sflow(D_{[0,1]}^{\psi, U}; P_t^U)_{0\leq t \leq \pi/4} + \int_0^{\pi/4} \frac{d}{dt}\frac{1}{2}\eta(D_{[0,1]}^{\psi, U}; P_t^U) dt.
\end{align*}
The formula now follows from the definition of $\overline \eta(\partial M, U)$ and the proposition above.

\end{proof}

\section{Relative Index Pairing for Odd Dimensional Manifolds with Boundary}\label{sec:dz}

In this section, we shall use the Toeplitz index theorem for odd dimensional manifolds with boundary by Dai and Zhang to prove our analogue of the index pairing formula by Lesch, Moscovici and Pflaum \cite[Theorem $7.6$]{L-M-P09}.

First let us recall the even case. Let $X$ be an even dimensional spin manifold with boundary $\partial X$. We assume its Riemannian metric has product structure near the boundary. The associated Dirac operator takes of the following form \[  D_X = \begin{pmatrix} & D^- \\ D^+ \end{pmatrix} = \begin{pmatrix} & -\frac{d}{dx} + D_{\partial X} \\ \frac{d}{dx} + D_{\partial X} \end{pmatrix} \]
near the boundary, where $D_{\partial X}$ is the Dirac operator over $\partial X$, cf. Appendix $\ref{app:spinor}$ .  
\begin{definition}
Let $P_{\geq 0} = \chi_{[0,\infty)}(D_{\partial X}) $ and $D^+_{P_{\geq 0}}$ be the elliptic operator $D^+$ with the APS boundary condition $P_{\geq 0}$, cf.\cite{A-P-S75a}. Then $ \ind_{APS}(D^+) := \ind(D^+_{P_{\geq 0}})$.
\end{definition}
Recall that a relative $K$-cycle in $K^0(X, \partial X)$ is a triple $ [p, q, h_s]$ such that $p, q \in M_n(C^\infty(X))$ are two projections over $X$ and $h_s \in M_n(C^\infty(\partial X))$, $s\in [0,1]$, is a path of projections over $\partial X$ such that $h_0 = p|_{\partial X}$ and $h_1 = q|_{\partial X}$. If $p$ and $q$ are constant along the normal direction near the boundary, then the relative index pairing by Lesch, Moscovici and Pflaum \cite[Theorem $7.6$]{L-M-P09} states that
\begin{align*}
	&\ind_{[D_X]}( [p, q, h_s]) = \ind_{APS} (qD^+ q) - \ind_{APS} (pD^+ p) +\sflow (h_s D_{\partial X} h_s)_{0\leq s \leq 1}.
\end{align*}



Now let $M$ be an odd dimensional spin manifold with boundary $\partial M$. We assume its Riemannian metric has product structure near the boundary. The Dirac operator $D$ over $M$ naturally induces an element in $KK(C_0(M\setminus \partial M), \cl_1) \cong K_1(M,\partial M )$ cf. \cite[Section 2]{B-D-T89}, from which one has the relative index pairing map
\begin{equation}\label{eq:relative}  
\ind_{[D]} : K^1(M, \partial M) \to \mathbb Z. 
\end{equation} As an intermediate step, let us first show a pairing formula by using the lifted data on $\mathbb S^1\times M$. The method of proof is similar to the one used in proving Theorem $\ref{Thm:main}$. Denote the Dirac operator over $\mathbb S^1\times M$ by $\ldirac$ and its restriction to the half-spinor bundles by $\ldirac^+$. We shall explain in detail the structure of $\ldirac $ near the boundary in appendix $\ref{app:spinor}$.

\begin{lemma}\label{lemma:oddindex}
	For a relative $K$-cycle $[U, V, u_s]\in K^1(M, \partial M)$, that is, $U, V\in U_n(C^\infty(M))$ are two unitaries over $M$ with $u_s \in U_n(C^\infty(\partial M))$, $s\in [0, 1]$, a smooth path of unitaries over $\partial M$ such that $u_0 = U|_{\partial M} $ and $u_1 = V|_{\partial M}$. If $U$ and $V$ are constant along the normal direction near the boundary, then 
	\begin{align*}  
	& \ind_{[D]} ([U,V, u_s]) \\
	&= \ind_{APS} (e_V\ldirac^+ e_V) - \ind_{APS} (e_U\ldirac^+ e_U) +\sflow (e_{u_s}D_{\mathbb S^1\times \partial M}  e_{u_s})_{0\leq s \leq 1}
	\end{align*}
\end{lemma}
\begin{proof}
A relative $K$-cycle $[U, V, u_s]\in K^1(M, \partial M)$ naturally induces a relative $K$-cycle $[e_U, e_V, e_{u_s}]\in K^0(M, \partial M) $. By \cite[Theorem $7.6$]{L-M-P09},
\begin{equation}
	\xymatrix{
	 [U,V, u_s] \ar@{|->}[d] \\
	\ind_{APS} (e_V\ldirac^+ e_V) - \ind_{APS} (e_U\ldirac^+ e_U) +\sflow (e_{u_s}D_{\mathbb S^1\times \partial M} e_{u_s})_{0\leq s \leq 1}
	}\label{Eq:oddpair}
\end{equation}
is a well-defined map from $K^1(M, \partial M) $ to $\mathbb Z$. We need to show that it does agree with the relative index pairing induced by that of $K^1(M\backslash \partial M)$. As before (cf. Section $ \ref{sec:pre}$ above), we can assume $U|_{[0, \epsilon)\times \partial M}=V|_{[0, \epsilon)\times \partial M }$ and $u_s = U|_{\partial M} = V|_{\partial M}$, for all $s\in [0,1]$. It suffices to prove the lemma for representatives of relative $K$-cycles of this special type. Notice that such a representative also defines an element in $K^1(M\backslash \partial M)$ by its restriction to $M\backslash \partial M$ and recall from Section $ \ref{sec:pre} $ that the index map $\eqref{eq:relative}$ has the following explicit formula:
\[ \ind_{[D]} ([V, U, u_s]) = -\int_{M} \ahat(M) \wedge \left[\Ch_\bullet(V) -\Ch_\bullet(U)\right]. \]
Now by the APS index theorem for manifolds with boundary, 
\begin{align} 
	\ind_{APS}(e_U\ldirac^+ e_U) &= \int_{\mathbb S^1 \times M} \ahat(\mathbb S^1\times M) \wedge \Ch_\bullet(e_U) - \reta(e_U D_{\mathbb S^1\times \partial M} e_U) \\
	& =  - \int_{M} \ahat(M) \wedge \Ch_\bullet(U) - \reta(e_U D_{\mathbb S^1\times \partial M} e_U). 
\end{align} 
where the second equality follows from Proposition $\ref{prop:product}$. There is a similar equation where we replace $U$ by $V$. It follows that the image of a representative of the special type as above, under the map $\eqref{Eq:oddpair}$, is equal to 
\[  -\int_{M} \ahat(M) \wedge \left[\Ch_\bullet(V) -\Ch_\bullet(U)\right].\]
This agrees with the relative index map $\eqref{eq:relative}$.
\end{proof}

Using this lemma and another two lemmas below, we shall now prove our main result in this section. 

\begin{theorem}\label{Thm:oddpairing}
	For a relative $K$-cycle $[U, V, u_s]\in K^1(M, \partial M)$, that is, $U, V\in U_n(C^\infty(M))$ are two unitaries over $M$ with $u_s \in U_n(C^\infty(\partial M))$, $s\in [0, 1]$, a smooth path of unitaries over $\partial M$ such that $u_0 = U|_{\partial M} $ and $u_1 = V|_{\partial M}$. If $U$ and $V$ are constant along the normal direction near the boundary, then 
	\[  \ind_{[D]}( [U,V, u_s]) = \ind(T_V) -\ind(T_U) +\sflow \left(u_s^{-1}D_{[0,1]} u_s; P_0^{u_s}\right)\]
	where $\sflow \left(u_s^{-1}D_{[0,1]} u_s; P_0^{u_s}\right)$ is the spectral flow of the path of elliptic operators $(u_s^{-1}D_{[0,1]} u_s; P_0^{u_s})$, $s\in [0, 1]$, with APS type boundary conditions $P_0^{u_s}$ as in $\eqref{eq:path}$.
\end{theorem}
\begin{proof}
By formula $\eqref{Eq:toep}$, we have 
\begin{align*}
	&\ind(T_V)  - \ind(T_U) \\
    & = -\int_M \ahat (M) \wedge \Ch_\bullet(V) - \reta(\partial M, V) + \tau_{\mu}(VP^\partial 
V^{-1}, P^\partial, \mathcal P_M) \\
    & \quad+\int_M \ahat (M) \wedge \Ch_\bullet(U) + \reta(\partial M, U) - \tau_{\mu}(UP^\partial U^{-1}, P^\partial, \mathcal P_M) \\
   & = -\int_M \ahat (M) \wedge [\Ch_\bullet(V) -\Ch_\bullet(U)] + \reta(\partial M, U) -\reta(\partial M, V) 
\end{align*}
since $\tau_{\mu}(UP^\partial U^{-1}, P^\partial, \mathcal P_M) = \tau_{\mu}(VP^\partial V^{-1}, P^\partial, \mathcal P_M)$ by \cite[Lemma 6.10]{PK-ML04}.
Notice that  
\begin{align*}
	&\quad \reta(\partial M, U) -\reta(\partial M, V) +\sflow \left(u_s^{-1}D_{[0,1]} u_s; P_0^{u_s}\right)_{0\leq s \leq 1}\\
	& = \reta(e_U D_{\mathbb S^1\times \partial M} e_U) - \sflow(D_{[0,1]}^{\psi, U}; P_t^U) - \sflow \left(D_{[0,1]}^{\psi, U}(t); P_0^U\right) \\
	&\quad - \reta(e_V D_{\mathbb S^1\times \partial M} e_V) + \sflow(D_{[0,1]}^{\psi, V}; P_t^V) + \sflow \left(D_{[0,1]}^{\psi, V}(t); P_0^V\right) \\
	& \quad + \sflow \left(u_s^{-1}D_{[0,1]} u_s; P_0^{u_s}\right)_{0\leq s \leq 1}
\end{align*}
which is equal to 
\[\reta(e_UD_{\mathbb S^1\times \partial M}e_U) -\reta(e_V D_{\mathbb S^1\times \partial M} e_V) + \sflow (e_{u_s}D_{\mathbb S^1\times\partial M}e_{u_s})_{0\leq s \leq 1}\]
by the lemmas below.
Hence 
\begin{align*}
& \ind(T_V) -\ind(T_U) +\sflow \left(u_s^{-1}D_{[0,1]} u_s; P_0^{u_s}\right)_{0\leq s \leq 1}\\
& = -\int_M \ahat (M) \wedge [\Ch_\bullet(V) -\Ch_\bullet(U)] \\
& \quad -\reta(e_V D_{\mathbb S^1\times \partial M} e_V) +\reta(e_U D_{\mathbb S^1\times \partial M} e_U) + \sflow (e_{u_s}  D_{\mathbb S^1\times \partial M} e_{u_s})_{0\leq s \leq 1}\\
& =  \ind_{APS} (e_V\ldirac^+ e_V) - \ind_{APS} (e_U\ldirac^+ e_U) +\sflow (e_{u_s} D_{\mathbb S^1\times \partial M}e_{u_s})_{0\leq s \leq 1}
\end{align*}
which is equal to $\ind_{[D]}([U, V, u_s])$ by Lemma $\ref{lemma:oddindex}$.
\end{proof}

\begin{lemma}
\begin{align*}
&\sflow\left(D^{\psi, u_s}_{[0,1]};  P_0^{u_s}\right)_{0\leq s \leq 1} \\
& =\sflow(D_{[0,1]}^{\psi, U}; P_t^U) - \sflow(D_{[0,1]}^{\psi, V}; P_t^V) + \sflow (e_{u_s}D_{\mathbb S^1\times \partial M}e_{u_s})_{0\leq s \leq 1}
\end{align*}
\end{lemma}
\begin{proof}
Consider the $(t, s)$-parametrized family of operators 
\[\left(D^{\psi, u_s}_{[0,1]}; P_t^{u_s}\right)_{(0\leq t \leq \pi/4 \ ; \ 0\leq s \leq 1)}\] where $P_t^{u_s}$
 is defined as in Eq. $\eqref{eq:path}$.
Note that \[  P_0^{u_s} = \begin{pmatrix} P^\partial & 0 \\ 0 & \Id-u_s^{-1}P^\partial u_s \end{pmatrix}\quad  \textup{and} \quad P_{\pi/4}^{u_s} =\frac{1}{2}\begin{pmatrix} 1 & -u_s\\ -u_s^{-1} & 1\end{pmatrix}.\]
Hence
\[ e_{u_s}D_{\mathbb S^1\times \partial M} e_{u_s} = ( D^{\psi, u_s}_{[0,1]} ;  P_{\pi/4}^{u_s}) .\]
Consider the following diagram
\[\xymatrixcolsep{7pc}\xymatrixrowsep{4pc}\xymatrix{
\left(D^{\psi, V}_{[0,1]}; P_0^V\right)  & \left(D^{\psi, V}_{[0,1]}; P_{\pi/4}^V\right) \ar@{~>}[l]_{\displaystyle \left(D^{\psi, V}_{[0,1]}; P_t^V\right)}\\
\left(D^{\psi, U}_{[0,1]}; P_0^U\right) \ar@{~>}[u]^{\displaystyle \left(D^{\psi, u_s}_{[0,1]}; P_0^{u_s}\right)}\ar@{~>}[r]_{\displaystyle \left(D^{\psi, U}_{[0,1]}; P_t^U\right)} & \left(D^{\psi, U}_{[0,1]}; P_{\pi/4}^U\right) \ar@{~>}[u]_{\displaystyle \left(D^{\psi, u_s}_{[0,1]}; P_{\pi/4}^{ u_s}\right)} 
}
\]
where the arrows stand for smooth paths connecting the corresponding vertices.
Now the lemma follows from the homotopy invariance of the spectral flow.
\end{proof}

Now let 	
\[ 	D^{\psi, u_s}_{[0,1]}(t) := D_{[0,1]} + (1-t\psi) u_s^{-1}\left[D_{[0,1]}, u_s\right],\] then  the same argument above  proves the following lemma.

\begin{lemma}
	\begin{align*}
	& \sflow\left( u_s^{-1}D_{[0,1]}u_s; P_0^{u_s}\right)_{0\leq s \leq 1}  \\
	&=	\sflow \left(D_{[0,1]}^{\psi, U}(t); P_0^U\right) - \sflow \left(D_{[0,1]}^{\psi, V}(t); P_0^V\right) + \sflow\left(D^{\psi, u_s}_{[0,1]};  P_0^{u_s}\right)_{0\leq s \leq 1}  
	\end{align*}

\end{lemma}
\begin{proof}
Consider the $(s, t)$-parametrized family of operators\[\left(D^{\psi, u_s}_{[0,1]}(t), P_0^{u_s} \right)_{0\leq t, s \leq 1}\] cf. the following diagram
	\[\xymatrixcolsep{7pc}\xymatrixrowsep{4pc}\xymatrix{
	\left(D^{\psi, V}_{[0,1]}(0); P_0^V\right)  & \left(D^{\psi, V}_{[0,1]}(1); P_0^V\right) \ar@{~>}[l]_{\displaystyle\left(D^{\psi, V}_{[0,1]}(t); P_0^V\right)}\\
	\left(D^{\psi, U}_{[0,1]}(0); P_0^U\right) \ar@{~>}[u]^{\displaystyle\left(D^{\psi, u_s}_{[0,1]}, P_0^{u_s}\right)}\ar@{~>}[r]_{\displaystyle\left(D^{\psi, V}_{[0,1]}(t); P_0^U\right)} & \left(D^{\psi, U}_{[0,1]}(1); P_0^U\right) \ar@{~>}[u]_{\displaystyle\left(D^{\psi, u_s}_{[0,1]}; P_0^{u_s}\right)} 
	}
	\]
Notice that $D^{\psi, u_s}_{[0,1]}(1) = D^{\psi, u_s}_{[0,1]}$ and $D^{\psi, u_s}_{[0,1]}(0) = u_s^{-1} D_{[0,1]}u_s$. The lemma follows by the homotopy invariance of the spectral flow.
\end{proof}

\appendix

\section{Spinor Bundles and Dirac on Manifolds with boundary}\label{app:spinor}
The material in this appendix is well known. The purpose is to clarify the relations among various Dirac operators arising in this article for the convenience of the reader.
Suppose $M$ is an odd dimensional spin manifold with boundary. Its Riemannian metric assumes a product structure near the boundary. Let $\mathcal S$ (resp. $\mathcal S_M$) be the spinor bundle over $\mathbb S^1 \times \partial M$ ( resp. $M$). Then $\textup{Cl}(T_{\partial M})$ the Clifford algebra over $\partial M$ is identified with the even part of $\textup{Cl}(T_{\mathbb S^1\times \partial M})$ the Clifford algebra over $\mathbb S^1\times \partial M$ by
\[ \cl^{\partial}(e_i) \mapsto \cl(e_i)\cdot \cl(d/d\theta).  \] 
where $\cl^{\partial}(\cdot)$, resp. $\cl(\cdot)$, is the Clifford multiplication on  $\mathcal S^{\partial}$, resp. $\mathcal S$. This way $\mathcal S|_{\partial M}$, the restriction of $\mathcal S$ to $\{0\}\times \partial M$,  is identified with $\mathcal S^{\partial} = \mathcal S^{\partial, +} \oplus \mathcal S^{\partial, -}$ the spinor bundle over $\partial M$.

Notice that $\widehat {\mathcal S}$, the spinor bundle over $[0, 1) \times \mathbb S^1 \times  \partial M$, is naturally isomorphic to $\mathbb C^2  \widehat{\otimes} \mathcal S^{\partial}$. Here $\mathbb C^2 = \mathbb C^+ \oplus \mathbb C^-$  and  $\widehat{\otimes}$ stands for graded tensor product. 
Denote the Dirac operator over $[0, 1)\times \mathbb S^1\times \partial M$ by $\ldirac$. Then
\[ \ldirac = \begin{pmatrix} 0 & -\frac{d}{dx}+ i\frac{d}{d\theta} \\ \frac{d}{dx}+ i\frac{d}{d\theta}   & 0\end{pmatrix} \widehat\otimes I_{\mathcal S^{\partial}} + I_{\mathbb C^2}\widehat\otimes D^\partial.\]
We identify $\textup{Cl}(T_{\mathbb S^1\times \partial M})$ the Clifford algebra over $\mathbb S^1\times \partial M$ with the even part of $\textup{Cl}(T_{\mathbb S^1\times M})$ the Clifford algebra over $\mathbb S^1\times M$ by
\[ \cl(e_i)  \mapsto \widehat\cl(e_i)\cdot \widehat\cl(d/dx)  \] 
for $e_i\in T_{\mathbb S^1\times \partial M}$, where $\widehat\cl(\cdot)$ is Clifford multiplication on $\widehat {\mathcal S}$. From this, one has
\begin{align*}
	 & \widehat {\mathcal S}^+ = \mathbb C^+\otimes \mathcal S^{\partial, +} \oplus \mathbb C^-\otimes \mathcal S^{\partial, -} \cong S^{\partial, +} \oplus \mathcal S^{\partial, -} \equiv \mathcal S\\
	 & \widehat{\mathcal S}^- = \mathbb C^-\otimes \mathcal S^{\partial, +} \oplus \mathbb C^+\otimes \mathcal S^{\partial, -} \cong \cl(d/dx)\widehat {\mathcal S}^+, 
\end{align*}

\begin{lemma}\label{Lemma:dirac}
With the idenfications of spinor bundles as above, 
		\begin{align*} \ldirac &= \begin{pmatrix} & -\frac{d}{dx} + D_{\mathbb S^1\times \partial M} \\ \frac{d}{dx} + 
			D_{\mathbb S^1\times \partial M} \end{pmatrix} \\
		       & = \begin{pmatrix} &  & -\frac{d}{dx}+ i\frac{d}{d\theta} & iD^\partial|_{ S^{\partial, -} }  \\
		                  &   & -iD^\partial|_{ S^{\partial, +} } & -\frac{d}{dx}- i\frac{d}{d\theta} \\
		                  \frac{d}{dx}+ i\frac{d}{d\theta} & iD^\partial|_{S^{\partial, + }} &  & \\
		                  -iD^\partial|_{S^{\partial, +}} & \frac{d}{dx} - i\frac{d}{d\theta}  &  &   \end{pmatrix}
		\end{align*}
where $D_{\mathbb S^1\times \partial M}$ (resp. $D^\partial$) is the Dirac operator over $\mathbb S^1\times \partial M$ (resp. $\partial M$). In particular, 
	\[  
D_{\mathbb S^1\times \partial M} = \cl(d/{d\theta})\left(\frac{d}{d\theta} + D^\partial\right)
	\]
with \[ \cl(d/d\theta) = \begin{pmatrix}  i &   \\
	                   & - i
	                  \end{pmatrix}\quad \textup{and} \quad D^\partial = \begin{pmatrix} &  D^\partial|_{\mathcal S^{\partial, -} }  \\
	                  D^\partial|_{\mathcal S^{\partial, +} } & \end{pmatrix}. \]
\end{lemma}
\begin{proof}
With the identification $\widehat {\mathcal S}^- = \widehat \cl(d/dx)\widehat {\mathcal S}^+$, one has
\begin{align*}
	-\widehat \cl(d/dx)\ldirac |_{\mathcal S^+} & = -\widehat \cl(d/dx) \left(\widehat \cl(d/dx)\frac{d}{dx} + \widehat \cl(d/d\theta)\frac{d}{d\theta}+ \sum_{i}\widehat \cl(e_i)\nabla_{e_i}\right)\\
	& = \frac{d}{dx}  -\widehat \cl(d/dx)\cdot \widehat \cl(d/d\theta)\frac{d}{d\theta} - \sum_{i}\widehat \cl(d/dx)\cdot \widehat \cl(e_i)\nabla_{e_i}\\
	& = \frac{d}{dx} + \cl(d/d\theta)\frac{d}{d\theta} + \sum_{i} \cl(e_i)\nabla_{e_i}\\
	& = \frac{d}{dx} + \cl(d/d\theta)\left(\frac{d}{d\theta} + \sum_{i} \cl^{\partial}(e_i)\nabla_{e_i}\right)
\end{align*}
Similarly, 
\begin{align*}
	\ldirac |_{\mathcal S^-} \widehat \cl(d/dx) & =  \left(\widehat \cl(d/dx)\frac{d}{dx} + \widehat \cl(d/d\theta)\frac{d}{d\theta}+ \sum_{i}\widehat \cl(e_i)\nabla_{e_i}\right)\widehat \cl(d/dx)\\
	& = -\frac{d}{dx} + \cl(d/d\theta)\left(\frac{d}{d\theta} + \sum_{i} \cl^{\partial}(e_i)\nabla_{e_i}\right)
\end{align*}
Notice that $ \cl(d/{d\theta})\left(\frac{d}{d\theta} + D^\partial\right)$ is the Dirac operator over $\mathbb S^1\times \partial M$, hence 
\[ \ldirac = \begin{pmatrix} & -\frac{d}{dx} + D_{\mathbb S^1\times \partial M} \\ \frac{d}{dx} + 
	D_{\mathbb S^1\times \partial M} \end{pmatrix}. \]
To finish the proof, one notices that
\begin{align*} \cl(d/d\theta) = \widehat \cl(d/d\theta) \cdot \widehat \cl(d/dx) &= \begin{pmatrix} 0 & i \\ i & 0   \end{pmatrix}\widehat\otimes I_{S^{\partial}} \cdot \begin{pmatrix} 0 &  -1 \\ 1 & 0  \end{pmatrix}\widehat\otimes I_{S^{\partial}}  \\ 
& =\begin{pmatrix} i&  0 \\ 0 & -i  \end{pmatrix}\widehat\otimes I_{S^{\partial}}
\end{align*}

\end{proof}

\end{document}